\documentclass[a4paper]{amsart}
\usepackage{amscd}
\usepackage{amsmath}
\usepackage{amssymb}
\usepackage{amsthm}
\usepackage{bbm}
\usepackage{stmaryrd}

\usepackage[T1]{fontenc}

\newif\ifpdf
\ifx\pdfoutput\undefined
   \pdffalse        
\else
   \pdfoutput=1     
   \pdftrue
\fi

\ifpdf
   \usepackage[pdftex]{graphicx}
\else
   \usepackage{graphicx}
\fi

\numberwithin{equation}{section} \swapnumbers

\newtheorem{satz}{Satz}[section]

\newtheorem{theorem}[satz]{Theorem}
\newtheorem{proposition}[satz]{Proposition}
\newtheorem{corollary}[satz]{Corollary}
\newtheorem{lemma}[satz]{Lemma}

\newtheorem{definition}[satz]{Definition}

\newtheorem{remark}[satz]{Remark}

\begin{document}

\hyphenation{rea-li-za-tion}

\title[Affine realizations for L\'evy term structure models]{Existence of affine realizations for L\'evy term structure models}
\author{Stefan Tappe}
\address{Leibniz Universit\"{a}t Hannover, Institut f\"{u}r Mathematische Stochastik, Welfengarten 1, 30167 Hannover, Germany}
\email{tappe@stochastik.uni-hannover.de}
\thanks{The author is grateful to Damir Filipovi\'c, Michael Kupper, Elisa Nicolato, Daniel Rost, David Skovmand, Josef Teichmann and Vilimir Yordanov for their helpful remarks and discussions.}
\thanks{The author is also grateful to two anonymous referees for their helpful comments and suggestions.}
\begin{abstract}
We investigate the existence of affine realizations for term
structure models driven by L\'evy processes. It turns out that we
obtain more severe restrictions on the volatility than in
the classical diffusion case without jumps. As special cases, we study constant direction volatilities and the existence of short rate realizations.
\end{abstract}
\maketitle

\section{Introduction}

A zero coupon bond with maturity $T$ is a financial asset which pays the holder one unit of cash at $T$. Its price at $t \leq T$ can be written as the continuous discounting of one unit of cash
\begin{align*}
P(t,T) = \exp \bigg( -\int_t^T f(t,s) ds \bigg),
\end{align*}
where $f(t,T)$ is the rate prevailing at time $t$ for instantaneous borrowing at time $T$, also called the forward rate for date $T$. The classical continuous framework for the evolution of the forward rates goes back to Heath, Jarrow and Morton (HJM) \cite{HJM}. They assume that, for every date $T$, the forward rates $f(t,T)$ follow an It\^o process of the form
\begin{align}\label{hjm-forward-rates-wiener}
df(t,T) = \alpha_{\rm HJM}(t,T) dt + \sigma(t,T) dW_t,\quad t \in [0,T]
\end{align}
where $W$ is a Wiener process.

In this paper, we consider L\'evy term structure models, which generalize the classical HJM
framework by replacing the Wiener process $W$ in (\ref{hjm-forward-rates-wiener}) by a more general L\'{e}vy process $X$, also taking into account the occurrence of jumps. This extension has been proposed by Eberlein et
al.\ \cite{Eberlein-Raible,Eberlein_O, Eberlein_J, Eberlein_K1,
Eberlein_K2,Eberlein_Kluge_Review}.
Other approaches in order to generalize the classical HJM framework
can be found in Bj\"ork et al.\ \cite{BKR,BKR0}, Carmona and
Tehranchi \cite{carteh}, and, e.g., \cite{Shirakawa, Jarrow_Madan,
Hyll}.

In the sequel, we therefore assume that, for every date $T$, the forward rates $f(t,T)$ follow an It\^o process
\begin{align*}
df(t,T) = \alpha_{\rm HJM}(t,T) dt + \sigma(t,T) dX_t,\quad t \in [0,T]
\end{align*}
with $X$ being a L\'{e}vy process. Note that such an HJM interest rate model is an infinite dimensional object, because for every date of maturity $T \geq 0$ we have an It\^{o} process.

There are several reasons why, in practice, we are interested in the existence of a finite dimensional realization, that is, the forward rate evolution being described by a finite dimensional state process. Such a finite dimensional realization ensures larger analytical tractability of the model, for example, in view of option pricing, see \cite{Duffie-Kan}. Moreover, as argued in \cite{bautei:05}, HJM models without a finite dimensional realization do not seem reasonable, because then the support of the forward rate curves $f(t,t+\cdot)$, $t > 0$ becomes too large, and hence any ``shape'' of forward rate curves, which we assume from the beginning to model the
market phenomena, is destroyed with positive probability.

For classical HJM models driven by a Wiener process, the construction of finite dimensional realizations for particular volatility structures has been treated in \cite{Jeffrey, Ritchken, Duffie-Kan, Bhar, Inui, Bj_Ch, Bj_Go,
Kwon_2001, Kwon_2003}, and finally, the problem concerning the existence of finite dimensional realizations has completely been solved in
\cite{Bj_Sv, Bj_La, Filipovic}, see also \cite{Filipovic-Teichmann-royal, Tappe-Wiener}. A survey about the topic can be found in \cite{Bjoerk}.

However, there are only very few references, such as \cite{Eberlein-Raible,
Kuechler-Naumann, Gapeev-Kuechler, Hyll}, that deal with affine
realizations for term structure models with jumps.

The purpose of the present paper is to investigate when a L\'evy driven term structure
model admits an affine realization.

The main idea is to switch to the Musiela parametrization of forward curves
$r_t(x) = f(t,t+x)$ (see \cite{Musiela}), and to consider the forward rates as the solution of
a stochastic partial differential equation (SPDE), the so-called
HJMM (Heath--Jarrow--Morton--Musiela) equation
\begin{align}\label{HJMM}
\left\{
\begin{array}{rcl}
dr_t & = & \big( \frac{d}{dx} r_t + \alpha_{\rm HJM}(r_t) \big) dt +
\sigma(r_{t-})dX_t
\medskip
\\ r_0 & = & h_0,
\end{array}
\right.
\end{align}
on a suitable Hilbert space $H$ of forward curves, where
$d / dx$ denotes the differential operator, which is generated
by the strongly continuous semigroup $(S_t)_{t \geq 0}$ of shifts. Such models have
been investigated in \cite{Filipovic-Tappe, P-Z-paper, Marinelli}.

The bank account $B$ is the riskless asset, which starts with one unit of cash and grows continuously at time $t$ with the short rate $r_t(0)$, i.e.
\begin{align*}
B(t) = \exp \bigg( \int_0^t r_s(0) ds \bigg), \quad t \geq 0.
\end{align*}
According to \cite{ds94}, the implied bond market, which we can now express as
\begin{align*}
P(t,T) = \exp \bigg( -\int_0^{T-t} r_t(x)dx \bigg), \quad 0 \leq t
\leq T
\end{align*}
is free of arbitrage if there exists an equivalent (local)
martingale measure $\mathbb{Q} \sim \mathbb{P}$ such that the discounted bond prices
$$
\frac{P(t,T)}{B(t)}, \quad t \in [0,T]
$$
are local $\mathbb{Q}$-martingales for all maturities $T$. If we formulate the
HJMM equation (\ref{HJMM}) with respect to such an equivalent
martingale measure $\mathbb{Q} \sim \mathbb{P}$, then the drift is determined by the
volatility, i.e. $\alpha_{\rm HJM} : H \rightarrow H$ in
(\ref{HJMM}) is given by the HJM drift condition
\begin{align}\label{alpha-HJM}
\alpha_{\rm HJM}(h) = \frac{d}{dx} \Psi \bigg(
-\int_0^{\bullet} \sigma(h)(\eta) d\eta \bigg) = - \sigma(h) \Psi'\bigg( -\int_0^{\bullet}
\sigma(h)(\eta) d\eta \bigg), \quad h \in H
\end{align}
where $\Psi$ denotes the cumulant generating function of the L\'evy
process, see \cite[Sec. 2.1]{Eberlein_O}.

As in \cite{Bj_Sv, Bj_La, Filipovic}, we can now regard the problem from a geometric point of view, i.e., the forward rate
process has to stay on a collection of finite dimensional affine manifolds indexed by the time $t$, a so-called foliation.

In general, invariance of a manifold for a stochastic process with jumps is a difficult issue, because we have to ensure that the process does not jump out of the manifold. This problem has been addressed in \cite{Kurtz-Pardoux-Protter}, where the authors consider a particular Stratonovich type integral (introduced by S.~I. Marcus, see \cite{Markus-1,Markus-2}) which, intuitively speaking, ensures that the jumps of a stochastic differential equation with vector fields being tangential to a given manifold $\mathcal{M}$, map the manifold $\mathcal{M}$ onto itself.

In the present paper, we avoid this problem by focusing on affine realizations, because for affine manifolds the jumps will automatically be captured, provided the volatility $h \mapsto \sigma(h)$ is tangential at each point of the manifold. Hence, in our framework, the stochastic integral in (\ref{HJMM}) is the usual It\^{o} integral.

Although the jumps of the L\'{e}vy process $X$ do not cause problems in this respect, that is, we do not have to worry that the solution $r$ jumps out of the manifold, our investigations will show -- and this is due to the particular structure of the HJM drift term $\alpha_{\rm HJM}$ in (\ref{alpha-HJM}) which ensures the absence of arbitrage -- that we obtain more severe restrictions
on the volatility $\sigma$ than in the classical diffusion case.

The remainder of this text is organized as follows: In Section
\ref{sec-foliations} we provide results on invariant foliations and
on affine realizations for SPDEs driven by L\'evy
processes. Afterwards, we introduce the term structure model in Section
\ref{sec-model}. After these preparations, in Sections \ref{sec-nec} and \ref{sec-suff} we present necessary and sufficient conditions
for the existence of affine realizations for L\'evy term structure models.
In Section \ref{sec-const} we study constant volatilities, and in Section \ref{sec-ddv} constant direction volatilities and consequences for the existence of short rate realizations. For the sake of lucidity, Appendix \ref{app-series} provides some auxiliary results that are needed in this text.

\section{Invariant foliations for SPDEs driven by L\'evy processes}\label{sec-foliations}

In this section, we provide results on invariant foliations for
SPDEs driven by L\'evy
processes, which we will apply to the HJMM equation (\ref{HJMM})
later on. The proofs of our results are similar to those from \cite[Sec. 2,3]{Tappe-Wiener}, where analogous statements for Wiener driven SPDEs are provided. Indeed, due to the affine structure of a foliation, the L\'{e}vy process cannot jump out of the foliation. We refer the reader to \cite[Sec. 2,3]{Tappe-Wiener} for more details and explanations about invariant foliations.

From now on, let $(\Omega,\mathcal{F},(\mathcal{F}_t)_{t \geq
0},\mathbb{P})$ be a filtered probability space satisfying the usual
conditions and let $X$ be a real-valued, square-integrable L\'evy process with Gaussian
part $c \geq 0$ and L\'evy measure $F$. In order to avoid
trivialities, we assume that $c + F(\mathbb{R}) > 0$.
Here, we shall deal with SPDEs of the type
\begin{align}\label{SPDE-manifold}
\left\{
\begin{array}{rcl}
dr_t & = & (A r_t + \alpha(r_t))dt + \sigma(r_{t-})dX_t
\medskip
\\ r_0 & = & h_0
\end{array}
\right.
\end{align}
on a separable Hilbert space $H$. In (\ref{SPDE-manifold}), the operator $A : \mathcal{D}(A) \subset H \rightarrow H$ is the infinitesimal
generator of a $C_0$-semigroup $(S_t)_{t \geq 0}$ on $H$, and $\alpha, \sigma : H \rightarrow H$ are measurable mappings. We refer to \cite{P-Z-book} for general information about SPDEs driven by L\'{e}vy processes.

In what follows, let $V \subset H$ be a finite dimensional linear
subspace.

\begin{definition}
A family $(\mathcal{M}_t)_{t \geq 0}$ of affine subspaces
$\mathcal{M}_t \subset H$, $t \geq 0$ is called a \emph{foliation
generated by $V$} if there exists $\psi \in C^1(\mathbb{R}_+;H)$
such that
\begin{align*}
\mathcal{M}_t = \psi(t) + V, \quad t \geq 0.
\end{align*}
\end{definition}

In what follows, let $(\mathcal{M}_t)_{t \geq 0}$ be a foliation
generated by the subspace $V$.

\begin{definition}\label{def-inv-foliation}
The foliation $(\mathcal{M}_t)_{t \geq 0}$ is called \emph{invariant for (\ref{SPDE-manifold})} if for every $t_0 \in
\mathbb{R}_+$ and $h \in \mathcal{M}_{t_0}$ there exists a weak solution $(r_t)_{t \geq 0}$ for (\ref{SPDE-manifold}) with $r_0 = h$ having c\`{a}dl\`{a}g sample paths such that
\begin{align*}
\mathbb{P}(r_t \in \mathcal{M}_{t_0 + t}) = 1 \quad \text{for all $t
\geq 0$.}
\end{align*}
\end{definition}

The Definition \ref{def-inv-foliation} of an invariant foliation slightly deviates from that in \cite{Tappe-Wiener}, as it includes the existence of a weak solution for (\ref{SPDE-manifold}). However, the proofs of the following results are similar to that in \cite{Tappe-Wiener}.

\begin{theorem}\label{thm-foliation-1}
We suppose that the following conditions are satisfied:
\begin{itemize}
\item The foliation $(\mathcal{M}_t)_{t \geq 0}$ is invariant
for (\ref{SPDE-manifold}).

\item The mappings $\alpha$ and $\sigma$ are continuous.
\end{itemize}
Then, for all $t \geq 0$ the following conditions hold true:
\begin{align}\label{domain-pre}
\mathcal{M}_t &\subset \mathcal{D}(A),
\\ \label{nu-pre} \nu ( h ) &\in T \mathcal{M}_t, \quad h \in
\mathcal{M}_t
\\ \label{sigma-pre} \sigma ( h ) &\in V, \quad h \in
\mathcal{M}_t.
\end{align}
\end{theorem}

In Theorem \ref{thm-foliation-1}, the mapping $\nu : \mathcal{D}(A) \rightarrow H$ is defined by $\nu := A + \alpha$, and $T \mathcal{M}_t$ denotes the tangent space of the foliation at time $t$, see \cite{Tappe-Wiener}.

\begin{theorem}\label{thm-foliation-2}
We suppose that the following conditions are satisfied:
\begin{itemize}
\item Conditions (\ref{domain-pre})--(\ref{sigma-pre}) hold true.

\item $\alpha$ and $\sigma$ are Lipschitz continuous.
\end{itemize}
Then, the foliation $(\mathcal{M}_t)_{t \geq 0}$ is invariant for (\ref{SPDE-manifold}).
\end{theorem}

The previous results lead to the following definition
of an affine realization:

\begin{definition}
Let $V \subset H$ be a finite dimensional subspace. 
\begin{enumerate}
\item The SPDE (\ref{SPDE-manifold}) has an {\em affine realization generated by
$V$}, if for each $h_0 \in \mathcal{D}(A)$ there exists a foliation
$(\mathcal{M}_t^{(h_0)})_{t \geq 0}$ generated by $V$ with $h_0 \in
\mathcal{M}_0^{(h_0)}$, which is invariant for (\ref{SPDE-manifold}).

\item In this case, we call $d := \dim V$ the {\em dimension} of the affine realization.

\item The SPDE (\ref{SPDE-manifold}) has an {\em affine realization}, if it has an affine realization generated by some subspace $V$.

\item An affine realization generated by some subspace $V$ is called \emph{minimal}, if for another affine realization generated by some subspace $W$ we have $V \subset W$.

\end{enumerate}
\end{definition}

\begin{lemma}\label{lemma-ddv}
Let $V \subset H$ be a finite dimensional subspace. We suppose that the following conditions are satisfied:
\begin{itemize}
\item The SPDE (\ref{SPDE-manifold}) has an affine realization generated by $V$.

\item $\alpha$ and $\sigma$ are continuous.
\end{itemize}
Then we have $\sigma(h) \in V$ for all $h \in H$.
\end{lemma}

\begin{proof}
Using Theorem \ref{thm-foliation-1} we have $\sigma(h) \in V$ for all $h \in \mathcal{D}(A)$. Since $\sigma$ is continuous, $\mathcal{D}(A)$ is dense in $H$, and $V$ is closed, we deduce that  $\sigma(h) \in V$ for all $h \in H$.
\end{proof}

\section{Presentation of the term structure model}\label{sec-model}

In this section, we shall introduce the L\'evy term structure model. Recall that $c \geq 0$ denotes the Gaussian part and $F$ the L\'evy measure of the L\'evy process $X$. We define the domain
\begin{align*}
\mathcal{D}(\Psi) := \bigg\{ z \in \mathbb{R} : \int_{\{ | x | > 1 \}} e^{zx} F(dx) < \infty \bigg\}
\end{align*}
and the cumulant generating function
\begin{align*}
\Psi : \mathcal{D}(\Psi) \rightarrow \mathbb{R}, \quad \Psi(z) := bz + \frac{c}{2} z^2 + \int_{\mathbb{R}} \big( e^{z x} - 1 - z
x \big) F(dx),
\end{align*}
where $b \in \mathbb{R}$ denotes the drift of $X$. Note that $\Psi$ is of class $C^{\infty}$ in the interior of $\mathcal{D}(\Psi)$.
In what follows, we assume that $K \subset \mathcal{D}(\Psi)$ for some compact interval $K$ with $0 \in {\rm Int} \, K$. Then, the cumulant generating function $\Psi$ is even analytic on the interior of $K$, and thus, for some $\epsilon > 0$ we obtain the power series representation
\begin{align}\label{Taylor-series-Psi}
\Psi(z) = \sum_{n=0}^{\infty} a_n z^n, \quad z \in (-\epsilon,\epsilon)
\end{align}
where the coefficients $(a_n)$ are given by
\begin{align*}
a_n = \frac{\Psi^{(n)}(0)}{n!}, \quad n \in \mathbb{N}_0.
\end{align*}
Note that
\begin{align}\label{psi-der-2}
a_2 = \frac{1}{2} \bigg( c + \int_{\mathbb{R}} x^2 F(dx) \bigg) \quad \text{and} \quad a_n = \frac{1}{n!} \int_{\mathbb{R}} x^n F(dx) \text{ for $n \geq 3$.}
\end{align}
We fix an arbitrary constant $\beta > 0$ and denote by $H_{\beta}$ the
space of all absolutely continuous functions $h : \mathbb{R}_+
\rightarrow \mathbb{R}$ such that
\begin{align}\label{def-norm}
\| h \|_{\beta} := \bigg( |h(0)|^2 + \int_{\mathbb{R}_+} |h'(x)|^2
e^{\beta x} dx \bigg)^{1/2} < \infty.
\end{align}
Spaces of this kind have been introduced in \cite{fillnm}. We also refer to \cite[Sec. 4]{Tappe-Wiener}, where some relevant properties have been summarized. Let $H_{\beta}^0$ be the subspace
\begin{align}\label{def-subspace}
H_{\beta}^0 := \Big\{ h \in H_{\beta} : \lim_{x \rightarrow \infty} h(x) = 0 \Big\}. 
\end{align}
We fix arbitrary constants $0 < \beta < \beta'$.

\begin{definition}
Let $H_{\beta,\beta'}^{\Psi}$ be the set of all mappings $\sigma : H_{\beta} \rightarrow H_{\beta'}^0$ such that
\begin{align*}
- \int_0^x \sigma(h)(\eta) d\eta \in K \quad \text{for all $h \in H_{\beta}$ and $x \in \mathbb{R}_+$.}
\end{align*}
\end{definition}

For a volatility $\sigma \in H_{\beta,\beta'}^{\Psi}$ we define the drift $\alpha_{\rm HJM}$ according to the HJM drift condition (\ref{alpha-HJM}).

\begin{remark}
Due to Lemma \ref{lemma-ddv}, throughout this text we will deal with volatility structures of the form
\begin{align}\label{ddv-pre}
\sigma(h) = \sum_{i=1}^p \Phi_i(h) \lambda_i, \quad h \in H_{\beta}
\end{align}
with real-valued mappings $\Phi_1,\ldots,\Phi_p : H_{\beta} \rightarrow \mathbb{R}$ and functions $\lambda_1,\ldots,\lambda_p \in H_{\beta'}^0$. By \cite[Lemma 4.3]{Tappe-Wiener} we have $\Lambda_1,\ldots,\Lambda_p \in H_{\beta}$, where $\Lambda_j := \int_0^{\bullet} \lambda(\eta) d\eta$ for $j=1,\ldots,p$, and hence, these functions are bounded. Thus, any volatility $\sigma$ of the form (\ref{ddv-pre}), for which the mappings $\Phi_1,\ldots,\Phi_p$ are suitably bounded, belongs to $H_{\beta,\beta'}^{\Psi}$.
\end{remark}

We denote by $(S_t)_{t \geq 0}$ the shift-semigroup on $H_{\beta}.$
From the theory of strongly continuous semigroups (see, e.g. \cite{Pazy}) it is well-known that the domain $\mathcal{D}(d/dx)$, endowed with the graph norm
\begin{align*}
\| h \|_{\mathcal{D}(d/dx)} := \big( \| h \|_{\beta}^2 + \| (d / dx) h \|_{\beta}^2 \big)^{1/2}, \quad h \in H_{\beta}
\end{align*}
itself is a separable Hilbert space, and that
$(S_t)_{t \geq 0}$ is also a $C_0$-semigroup on $(\mathcal{D}(d/dx), \| \cdot \|_{\mathcal{D}(d/dx)})$.
Using similar techniques as in \cite[Sec. 4]{Tappe-Wiener} and \cite[Sec. 4]{Filipovic-Tappe}, we obtain the following auxiliary result.

\begin{lemma}\label{lemma-Psi-Lipschitz}
Let $\sigma \in H_{\beta,\beta'}^{\Psi}$ be arbitrary.
\begin{enumerate}
\item We have $\alpha_{\rm HJM}(h) \in H_{\beta}^0$ for all $h \in H_{\beta}$.

\item If $\sigma$ is continuous, then $\alpha_{\rm HJM}$ is continuous, too.

\item If $\sigma$ is Lipschitz continuous and bounded, then $\alpha_{\rm HJM}$ is Lipschitz continuous.

\item If $\sigma(\mathcal{D}(d/dx)) \subset \mathcal{D}(d/dx)$ and $\sigma$ is Lipschitz continuous and bounded on $(\mathcal{D}(d/dx), \| \cdot \|_{\mathcal{D}(d/dx)})$, then $\alpha_{\rm HJM}(\mathcal{D}(d/dx)) \subset \mathcal{D}(d/dx)$ and $\alpha_{\rm HJM}$ is Lipschitz continuous on $(\mathcal{D}(d/dx), \| \cdot \|_{\mathcal{D}(d/dx)})$.
\end{enumerate}
\end{lemma}

Note that the HJMM equation (\ref{HJMM})
is a particular example of the SPDE (\ref{SPDE-manifold}) on the state space $H = H_{\beta}$ with infinitesimal generator $A = d/dx$ and $\alpha = \alpha_{\rm HJM}$. Due to Lemma \ref{lemma-Psi-Lipschitz}, we can apply all previous results about invariant foliations from Section \ref{sec-foliations} in the sequel.

\section{Necessary conditions for the existence of affine realizations}\label{sec-nec}

In this section, we shall derive necessary conditions for the existence 
of affine realizations for L\'{e}vy term structure models. 

Throughout this section, we assume that the HJMM equation (\ref{HJMM}) has an affine realization generated by some subspace $V \subset H_{\beta}$. 
We suppose that the volatility $\sigma \in H_{\beta,\beta'}^{\Psi}$ is continuous. According to Lemma \ref{lemma-Psi-Lipschitz}, the drift $\alpha_{\rm HJM}$ is continuous, too.
Recall that $F$ denotes the L\'{e}vy measure of the driving L\'{e}vy process $X$ in (\ref{HJMM}). We suppose there exists an index $n_0 \in \mathbb{N}$ such that
\begin{align}\label{cond-f-n0}
\int_{\mathbb{R}} x^n F(dx) \neq 0 \quad \text{for all $n \geq n_0$,}
\end{align}
and we suppose that for each $\lambda \in V$ with $\lambda \neq 0$ we have
\begin{align}\label{cond-lambda-lin}
\lambda|_{[0,\kappa]} \not\equiv 0 \quad \text{for all $\kappa > 0$.}
\end{align}
We fix an arbitrary $h_0 \in \mathcal{D}(d/dx)$ and define the linear space $W := \langle \sigma(h_0 + V) \rangle$. Recall that a function $v \in \mathcal{D}((d/dx)^{\infty})$ is called \emph{quasi-exponential}, if
\begin{align*}
\dim \langle (d/dx)^n v : n \in \mathbb{N}_0 \rangle < \infty.
\end{align*}

\begin{theorem}\label{thm-ddv-pre}
The following statements are true:
\begin{enumerate}
\item We have $\sigma(H_{\beta}) \subset V$.

\item For every subspace $U \subset W$ with $\dim U \geq 1$ and each set $Y \subset \sigma(h_0 + V)$ with $Y \cap U \neq \emptyset$, the set $Y \cap U$ cannot be open in $U$.

\item If $\sigma$ is constant on $h_0 + V$, then each $v \in V$ is quasi-exponential, and we have $\langle (d/dx)^n v : n \in \mathbb{N}_0 \rangle \subset V$.
\end{enumerate}
\end{theorem}

\begin{remark}
The relation $\sigma(H_{\beta}) \subset V$ implies that the volatility $\sigma$ is of the form
\begin{align*}
\sigma(h) = \sum_{i=1}^p \Phi_i(h) \lambda_i, \quad h \in H_{\beta}
\end{align*}
with real-valued mappings $\Phi_1,\ldots,\Phi_p : H_{\beta} \rightarrow \mathbb{R}$ and functions $\lambda_1,\ldots,\lambda_p \in H_{\beta'}^0$. Theorem \ref{thm-ddv-pre} shows that we obtain restrictions on the mappings $\Phi_1,\ldots,\Phi_p$, which mean that their range cannot be arbitrarily rich. Such restrictions do not occur in the Wiener driven case, see, e.g. \cite{Bj_Sv, Bj_La, Filipovic, Filipovic-Teichmann-royal, Tappe-Wiener}.
\end{remark}

Before we start with the proof of Theorem \ref{thm-ddv-pre}, we shall derive some immediate consequences. If the volatility $\sigma$ is locally linear, then it vanishes. More precisely:

\begin{corollary}
Suppose there exist a linear operator $S \in L(V,W)$ and a nonempty open subset $O \subset V$ such that $\sigma(h_0 + v) = Sv$ for all $v \in O$. Then we have $S = 0$.
\end{corollary}

\begin{proof}
Setting $Y := \sigma(h_0+O) = S(O) \subset {\rm ran} \, S$, we have $Y \subset \sigma(h_0 + V)$ and, by the open mapping theorem, the range $Y$ is open in ${\rm ran} \, S$. Using Theorem \ref{thm-ddv-pre}, it follows that $S = 0$.
\end{proof}

The next corollary concerns the case of constant direction volatility:

\begin{corollary}\label{cor-ddv-pre}
If $\dim W = 1$, then $\sigma$ is constant on $h_0 + V$.
\end{corollary}

\begin{proof}
There exists $\lambda \in W$ with $W = \langle \lambda \rangle$. Suppose that $\sigma$ is not constant on $h_0 + V$. Then, there exist $a,b \in \mathbb{R}$ with $a < b$ and $a \lambda, b \lambda \in \sigma(h_0+V)$. The set $Y := \{ \theta \lambda : \theta \in (a,b) \} \subset W$ is open in $W$, and by the continuity of $\sigma$ we obtain $Y \subset \sigma(h_0+V)$,
which contradicts Theorem \ref{thm-ddv-pre}.
\end{proof}

\begin{remark}
The assumption $\dim W = 1$ implies that on $h_0 + V$ the volatility $\sigma$ is of the form $\sigma(h) = \Phi(h) \lambda$ with a real-valued mapping $\Phi$ and a function $\lambda \in H_{\beta'}^0$, whence we speak about constant direction volatility. As we shall see in Section \ref{sec-ddv}, in this particular situation we can replace (\ref{cond-f-n0}) by the weaker condition $F(\mathbb{R}) \neq 0$, and condition (\ref{cond-lambda-lin}) can be skipped.
\end{remark}

Our goal for the rest of this section is the proof of Theorem \ref{thm-ddv-pre}. The first statement of Theorem \ref{thm-ddv-pre} immediately follows from Lemma \ref{lemma-ddv}. According to \cite[Lemma 4.3]{Tappe-Wiener}, the integral operator
\begin{align*}
T : H_{\beta'}^0 \rightarrow H_{\beta}, \quad T\lambda := -\int_0^{\bullet} \lambda(\eta) d\eta
\end{align*}
is a bounded linear operator, and it is injective. We define the mapping $\Sigma := T \circ \sigma : H_{\beta} \rightarrow H_{\beta}$.

\begin{lemma}
We have $V \subset \mathcal{D}(d/dx)$, and there exists $g_0 \in H_{\beta}$ such that
\begin{align}\label{eqn-lin-ind}
\frac{d}{dx} v + \frac{d}{dx} \Psi ( \Sigma(h_0 + v) ) + g_0 \in V \quad \text{for all $v \in V$.}
\end{align}
\end{lemma}

\begin{proof}
We apply Theorem \ref{thm-foliation-1} to the invariant foliations $(\mathcal{M}_t^{(0)})_{t \geq 0}$ and
$(\mathcal{M}_t^{(h_0)})_{t \geq 0}$, implying $V \subset \mathcal{D}(d/dx)$ and the existence of some $h_0' \in H_{\beta}$ such that
\begin{align}\label{nu-one-dim}
\nu ( h_0 + v ) + h_0' \in V \quad \text{for all $v \in V$.}
\end{align}
Inserting the HJM drift condition (\ref{alpha-HJM}) into (\ref{nu-one-dim}), gives us relation (\ref{eqn-lin-ind}).
\end{proof}

Now, the third statement of Theorem \ref{thm-ddv-pre} is a direct consequence of relation (\ref{eqn-lin-ind}).

\begin{remark}\label{remark-U}
Integrating (\ref{eqn-lin-ind}), we see that the linear space
\begin{align}\label{def-V-h0}
V^{\Psi} := \langle \Psi ( \Sigma(h_0 + v) ) : v \in V \rangle
\end{align}
must necessarily be finite dimensional. In the present situation, by (\ref{Taylor-series-Psi}), (\ref{psi-der-2}) the cumulant generating function $\Psi$ is no polynomial, and hence, this condition is difficult to ensure without $\sigma$ being constant on $h_0 + V$.
\end{remark}

Note that for the proof of Theorem \ref{thm-ddv-pre} we have not used conditions (\ref{cond-f-n0}), (\ref{cond-lambda-lin}) up to this point.
We shall now prove the second statement of Theorem \ref{thm-ddv-pre}. In the sequel, for $z_0 \in \mathbb{R}^n$ and $\delta > 0$ we denote by $B_{\delta}(z_0)$ the open ball
\begin{align*}
B_{\delta}(z_0) := \{ z \in \mathbb{R}^n : \| z - z_0 \|_{\mathbb{R}^n} < \delta \}.
\end{align*}

\begin{proof}
(of the second statement of Theorem \ref{thm-ddv-pre}) Suppose there are a subspace $U \subset W$ with $\dim U \geq 1$ and a set $Y \subset \sigma(h_0 + V)$ with $Y \cap U \neq \emptyset$ such that $Y \cap U$ is open in $U$.
We will derive the contradiction
\begin{align}\label{U-contradiction}
\dim V^{\Psi} = \infty.
\end{align}
In order to prove (\ref{U-contradiction}), by virtue of (\ref{Taylor-series-Psi}), (\ref{psi-der-2}) and (\ref{cond-f-n0}) we may assume that $a_n \neq 0$ for all $n \in \mathbb{N}$. We set $E := T(U)$. Since $T$ is injective, we have $\dim E = \dim U \geq 1$. By the open mapping theorem, the set $T(Y \cap U)$ is open in $E$. Since $T(Y \cap U) \subset \Sigma(h_0 + V)$, there exist a direct sum decomposition $E = E_1 \oplus E_2$ with $\dim E_1 \geq 1$, elements $\Lambda_1 \in E_1$, $\Lambda_2 \in E_2$ with $\Lambda_1 \neq 0$, and constants $a,b \in \mathbb{R}$ with $a < b$ such that
\begin{align}\label{choose-from-range}
\theta \Lambda_1 + \Lambda_2 \in \Sigma(h_0 + V) \quad \text{for all $\theta \in (a,b)$.}
\end{align}
Now, let $m \in \mathbb{N}$ be arbitrary. By (\ref{choose-from-range}) there exist $\theta_1,\ldots,\theta_m \in (a,b)$ with $\theta_i \neq \theta_j$ for $i \neq j$ such that
\begin{align*}
\theta_i \Lambda_1 + \Lambda_2 \in \Sigma(h_0 + V) \quad \text{for all $i=1,\ldots,m$.}
\end{align*}
We will show that
\begin{align}\label{dim-equal-m}
\dim \langle \Psi(\theta_i \Lambda_1 + \Lambda_2) : i=1,\ldots,m \rangle = m.
\end{align}
Indeed, let $\xi_1,\ldots,\xi_m \in \mathbb{R}$ be such that
\begin{align}\label{linear-comb}
\sum_{i=1}^m \xi_i \Psi( \theta_i \Lambda_1 + \Lambda_2 ) = 0.
\end{align}
By the power series representation (\ref{Taylor-series-Psi}) there exists $\eta > 0$ such that for all $(y,z) \in B_{\eta}(0)$ we obtain
\begin{align*}
&\sum_{i=1}^m \xi_i \Psi(\theta_i y + z) = \sum_{i=1}^m \xi_i \sum_{n=0}^{\infty} a_n ( \theta_i y + z )^n
\\ &= \sum_{n=0}^{\infty} a_n \sum_{i=1}^m \xi_i ( \theta_i y + z )^n = \sum_{n=0}^{\infty} a_n \sum_{i=1}^m \xi_i \sum_{\genfrac{}{}{0pt}{}{k,l \in \mathbb{N}_0}{k+l=n}} (\theta_i y)^k z^l
\\ &= \sum_{n=0}^{\infty} a_n \sum_{\genfrac{}{}{0pt}{}{k,l \in \mathbb{N}_0}{k+l=n}} \bigg( \sum_{i=1}^m \xi_i \theta_i^k \bigg) y^k z^l.
\end{align*}
Hence, defining the coefficients
\begin{align}\label{def-c-k-l}
c_{(k,l)} := a_{k+l} \sum_{i=1}^m \xi_i \theta_i^k, \quad (k,l) \in \mathbb{N}_0^2,
\end{align}
there is a bijection $\pi : \mathbb{N}_0 \rightarrow \mathbb{N}_0^2$ such that the power series
\begin{align}\label{series-k-l}
\sum_{\genfrac{}{}{0pt}{}{n=0}{(k,l) = \pi(n)}}^{\infty} c_{(k,l)} y^k z^l
\end{align}
converges for all $(y,z) \in B_{\eta}(0)$. According to Proposition \ref{prop-power-series}, there exists $r > 0$ such that the power series (\ref{series-k-l}) converges absolutely and uniformly on $K_r(0)$ -- which denotes the compact ball defined in (\ref{comp-ball}) -- to a continuous function
\begin{align*}
f : K_r(0) \rightarrow \mathbb{R}, \quad f(y,z) = \sum_{(k,l) \in \mathbb{N}_0^2} c_{(k,l)} y^k z^l.
\end{align*}
We claim that
\begin{align}\label{all-c-k0-zero}
c_{(k,0)} = 0 \quad \text{for all $k \in \mathbb{N}_0$.}
\end{align}
Indeed, suppose that (\ref{all-c-k0-zero}) is not satisfied.
Then, there exists $k_0 \in \mathbb{N}_0$ such that $c_{(k_0,0)} \neq 0$ and $c_{(k,0)} = 0$ for $k < k_0$. Since $a_n \neq 0$ for all $n \in \mathbb{N}_0$, by (\ref{def-c-k-l}) for all $k < k_0$ and $l \in \mathbb{N}_0$ we obtain
\begin{align*}
c_{(k,l)} = a_{k+l} \sum_{i=1}^m \xi_i \theta_i^k = \frac{a_{k+l}}{a_{k}} c_{(k,0)} = 0.
\end{align*}
Since the power series (\ref{series-k-l}) converges absolutely for all $(y,z) \in B_{\eta}(0)$, we deduce that for some bijection $\tau : \mathbb{N}_0 \rightarrow \mathbb{N}_0^2$ the power series
\begin{equation}\label{series-k-l-2}
\begin{aligned}
\sum_{\genfrac{}{}{0pt}{}{n=0}{(k,l) = \tau(n)}}^{\infty} c_{(k_0+k,l)} y^k z^l &= \frac{1}{y^{k_0}} \sum_{\genfrac{}{}{0pt}{}{n=0}{(k,l) = \tau(n)}}^{\infty} c_{(k_0+k,l)} y^{k_0+k} z^l 
\\ &= \frac{1}{y^{k_0}} \sum_{\genfrac{}{}{0pt}{}{n=0}{(k,l) = \pi(n)}}^{\infty} c_{(k,l)} y^k z^l
\end{aligned}
\end{equation}
also converges for all $(y,z) \in B_{\eta}(0)$ with $y \neq 0$.
According to Proposition \ref{prop-power-series}, the power series (\ref{series-k-l-2}) converges absolutely and uniformly on $K_r(0)$ to a continuous function
\begin{align*}
g : K_r(0) \rightarrow \mathbb{R}, \quad g(y,z) = \sum_{(k,l) \in \mathbb{N}_0^2} c_{(k_0+k,l)} y^k z^l.
\end{align*}
Moreover, for each $(y,z) \in K_r(0)$ with $y \neq 0$ we have
\begin{align*}
g(y,z) = \frac{f(y,z)}{y^{k_0}}.
\end{align*}
Setting $\Lambda := (\Lambda_1,\Lambda_2)$, by (\ref{linear-comb}) we have
\begin{align}\label{f-zero}
f(y,z) = 0 \quad \text{for all $(y,z) \in \Lambda(\mathbb{R}_+) \cap K_r(0)$.}
\end{align}
Since $\Lambda_1 \neq 0$, by (\ref{cond-lambda-lin}) there exists a sequence $(x_n)_{n \in \mathbb{N}} \subset (0,\infty)$ with $x_n \rightarrow 0$ and $\Lambda_1(x_n) \neq 0$ for all $n \in \mathbb{N}$.
Since $\Lambda$ is continuous with $\Lambda(0) = (0,0)$, setting $(y_n,z_n) := \Lambda(x_n)$, $n \in \mathbb{N}$, we have $(y_n,z_n) \rightarrow (0,0)$. 
By (\ref{f-zero}) we obtain the contradiction
\begin{align*}
c_{(k_0,0)} = g(0,0) = \lim_{n \rightarrow \infty} g(y_n,z_n) = \lim_{n \rightarrow \infty} \frac{f(y_n,z_n)}{y_n^{k_0}} = 0.
\end{align*}
Consequently, we have (\ref{all-c-k0-zero}). Since $a_n \neq 0$ for all $n \in \mathbb{N}$, by the Definition (\ref{def-c-k-l}) we get
\begin{align*}
\sum_{i=1}^m \xi_i \theta_i^{k} = 0 \quad \text{for all $k \in \mathbb{N}_0$.}
\end{align*}
It follows that
$B \xi = 0$, where $B \in \mathbb{R}^{m \times m}$ denotes the Vandermonde matrix $B_{ki} = \theta_i^k$ for $k = 0,\ldots,m-1$ and $i = 1,\ldots,m$. Since $\theta_i \neq \theta_j$ for $i \neq j$, we deduce that $\xi_1 = \ldots = \xi_m = 0$, which proves (\ref{dim-equal-m}). Since $m \in \mathbb{N}$ was arbitrary, we obtain (\ref{U-contradiction}), which contradicts Remark \ref{remark-U}. This completes the proof of Theorem \ref{thm-ddv-pre}.
\end{proof}

\section{Sufficient conditions for the existence of affine realizations}\label{sec-suff}

In this section, we shall derive sufficient conditions for the existence 
of affine realizations for L\'{e}vy term structure models.

We suppose that the volatility $\sigma \in H_{\beta,\beta'}^{\Psi}$ is Lipschitz continuous and bounded. According to Lemma \ref{lemma-Psi-Lipschitz}, the drift $\alpha_{\rm HJM}$ is Lipschitz continuous, too.

We have seen that for the existence of an affine realization the linear spaces $V^{\Psi}$ defined in (\ref{def-V-h0}) must necessarily be finite dimensional. As discussed in Remark \ref{remark-U} (and shown in Theorem \ref{thm-ddv-pre}), this is difficult to ensure with a driving L\'{e}vy process having jumps, unless the volatility $\sigma$ is constant on the affine spaces generated the realization. Therefore, and because of Theorem \ref{thm-ddv-pre}, we make the following assumptions:
\begin{itemize}
\item There exists a finite dimensional subspace $W \subset H_{\beta}$ with $\sigma(H_{\beta}) \subset W$.

\item Each $w \in W$ is quasi-exponential. Then, the linear space
\begin{align*}
V := \langle (d/dx)^n w : w \in W \text{ and } n \in \mathbb{N}_0 \rangle
\end{align*}
is finite dimensional.

\item For each $h_0 \in \mathcal{D}(d/dx)$ the volatility $\sigma$ is constant on $h_0 + V$.
\end{itemize}

\begin{theorem}\label{thm-suff}
If the previous conditions are satisfied, then the HJMM equation (\ref{HJMM}) has a minimal realization generated by $V$.
\end{theorem}

\begin{proof}
Let $h_0 \in \mathcal{D}(d/dx)$ be arbitrary. Since $V \subset \mathcal{D}(d/dx)$ and $\sigma(H_{\beta}) \subset W$, we have $\sigma(\mathcal{D}(d/dx)) \subset \mathcal{D}(d/dx)$ and $\sigma$ is Lipschitz continuous and bounded on $(\mathcal{D}(d/dx), \| \cdot \|_{\mathcal{D}(d/dx)})$. By Lemma \ref{lemma-Psi-Lipschitz}, we have $\alpha_{\rm HJM}(\mathcal{D}(d/dx)) \subset \mathcal{D}(d/dx)$, and $\alpha_{\rm HJM}$ is Lipschitz continuous on $(\mathcal{D}(d/dx), \| \cdot \|_{\mathcal{D}(d/dx)})$. Thus, according to \cite[Thm. 6.1.7]{Pazy}, there exists a classical solution $\psi \in C^1(\mathbb{R}_+;H_{\beta})$ with $\psi(\mathbb{R}_+) \subset \mathcal{D}(d/dx)$ of the evolution
equation
\begin{align*}
\left\{
\begin{array}{rcl}
\frac{d}{dt} \psi(t) & = & \frac{d}{dx} \psi(t) + \alpha_{\rm HJM}(\psi(t))
\medskip
\\ \psi(0) & = & h_0.
\end{array}
\right.
\end{align*}
Defining the foliation $(\mathcal{M}_t^{(h_0)})_{t \geq 0}$ by $\mathcal{M}_t^{(h_0)} := \psi(t) + V$, relation (\ref{domain-pre}) is fulfilled, and we have (\ref{sigma-pre}), because $\sigma(H_{\beta}) \subset W$. Let $t \geq 0$ and $v \in V$ be arbitrary. By the HJM drift condition (\ref{alpha-HJM}), the drift $\alpha_{\rm HJM}$ is constant on $\psi(t) + V$, and hence, we obtain
\begin{align*}
\nu(\psi(t) + v) &= \frac{d}{dx} \psi(t) + \frac{d}{dx} v + \alpha_{\rm HJM}(\psi(t) + v) 
\\ &= \frac{d}{dt} \psi(t) - \alpha_{\rm HJM}(\psi(t)) + \frac{d}{dx} v + \alpha_{\rm HJM}(\psi(t))
\\ &= \frac{d}{dt} \psi(t) + \frac{d}{dx} v \in \frac{d}{dt} \psi(t) + V = T \mathcal{M}_t,
\end{align*}
showing (\ref{nu-pre}). Theorem \ref{thm-foliation-2} applies and yields that
the foliation $(\mathcal{M}_t^{(h_0)})_{t \geq 0}$ is invariant for the HJMM equation (\ref{HJMM}). Consequently, the HJMM equation (\ref{HJMM}) has an affine realization generated by $V$. The minimality follows from Theorem \ref{thm-ddv-pre}.
\end{proof}

\begin{remark}
In particular, for every volatility structure of the form
\begin{align*}
\sigma(h) = \sum_{i=1}^p \Phi_i(h) \lambda_i, \quad h \in H_{\beta}
\end{align*}
with quasi-exponential functions $\lambda_1,\ldots,\lambda_p \in H_{\beta'}^0$, the HJMM equation (\ref{HJMM}) has a minimal realization generated by
\begin{align*}
V = \langle (d/dx)^n \lambda_1 : n \in \mathbb{N}_0 \rangle + \ldots + \langle (d/dx)^n \lambda_p : n \in \mathbb{N}_0 \rangle,
\end{align*}
provided that for each $h_0 \in \mathcal{D}(d/dx)$ the mappings $\Phi_1,\ldots,\Phi_p : H_{\beta} \rightarrow \mathbb{R}$ are constant on the affine space $h_0 + V$.
\end{remark}

\section{Constant volatility}\label{sec-const}

In this section, we apply our previous results for the particular case of a constant volatility $\sigma \in H_{\beta'}^0$ with $\sigma \neq 0$. 

\begin{corollary}\label{cor-det}
The following statements are equivalent:
\begin{enumerate}
\item The HJMM equation (\ref{HJMM}) has an affine realization.

\item $\sigma$ is quasi-exponential.
\end{enumerate}
In either case, the HJMM equation (\ref{HJMM}) has a minimal realization generated by $V = \langle (d/dx)^n \sigma : n \in \mathbb{N}_0 \rangle$.
\end{corollary}

\begin{proof}
This is an immediate consequence of Theorems \ref{thm-ddv-pre} and \ref{thm-suff}.
\end{proof}

\begin{remark}
Consequently, for constant volatility
structures we obtain exactly the same criterion for the existence of an
affine realization as in the classical diffusion case, where the
HJMM equation (\ref{HJMM}) is driven by a Wiener process, namely the function $\lambda$ has to be quasi-exponential; see e.g. \cite{Bj_Sv, Bj_La, Tappe-Wiener}.
\end{remark}

\section{Constant direction volatility}\label{sec-ddv}

In this section, we shall tighten the statement of Corollary \ref{cor-ddv-pre} and present some consequences.

Throughout this section, we assume that the HJMM equation (\ref{HJMM}) has an affine realization generated by some subspace $V \subset H_{\beta}$. 
We suppose that the volatility $\sigma \in H_{\beta,\beta'}^{\Psi}$ is continuous. According to Lemma \ref{lemma-Psi-Lipschitz}, the drift $\alpha_{\rm HJM}$ is continuous, too. Moreover, we assume that $\dim W = 1$, where $W := \langle \sigma(H_{\beta}) \rangle$, and that $F(\mathbb{R}) \neq 0$, where $F$ denotes the L\'{e}vy measure of the driving L\'{e}vy process $X$ in (\ref{HJMM}).

\begin{theorem}\label{thm-ddv}
The following statements are true:
\begin{enumerate}
\item For each $h_0 \in H_{\beta}$ the volatility $\sigma$ is constant on the affine space $h_0 + V$.

\item Each $v \in V$ is quasi-exponential, and we have $\langle (d/dx)^n v : n \in \mathbb{N}_0 \rangle \subset V$.
\end{enumerate}
\end{theorem}

\begin{proof}
Let $h_0 \in \mathcal{D}(d/dx)$ be arbitrary. Suppose that $\sigma$ is not constant on $h_0 + V$. We will derive the contradiction
\begin{align}\label{contra-2}
\dim V^{\Psi} = \infty,
\end{align}
where the linear space $V^{\Psi}$ was defined in (\ref{def-V-h0}).
Since $T$ is injective, we have $\dim T(W) = 1$ with $T(W) = \langle \Sigma(H_{\beta}) \rangle$, and the mapping $\Sigma$ is not constant on $h_0 + V$. There exists $\Lambda \in T(W)$ with $T(W) = \langle \Lambda \rangle$. Since $\Sigma$ is not constant on $h_0 + V$, there exist $a,b \in \mathbb{R}$ with $a < b$ and $a \Lambda, b \Lambda \in \Sigma(h_0 + V)$. By the continuity of $\Sigma$ we obtain
\begin{align}\label{intervalue}
\theta \Lambda \in \Sigma(h_0 + V) \quad \text{for all $\theta \in [a,b]$.}
\end{align}
Now, let $m \in \mathbb{N}$ be arbitrary. By (\ref{intervalue}) there exist $\theta_1,\ldots,\theta_m \in [a,b]$ with $|\theta_i| \neq |\theta_j|$ for $i \neq j$ such that
\begin{align*}
\theta_i \Lambda \in \Sigma(h_0 + V) \quad \text{for all $i=1,\ldots,m$.}
\end{align*}
We will show that
\begin{align}\label{dim-m}
\dim \langle \Psi(\theta_i \Lambda) \rangle = m.
\end{align}
Indeed, let $\xi_1,\ldots,\xi_m \in \mathbb{R}$ be such that
\begin{align}
\sum_{i=1}^m \xi_i \Psi(\theta_i \Lambda) = 0.
\end{align}
By the power series representation (\ref{Taylor-series-Psi}) there exists $\eta > 0$ such that
\begin{align*}
\sum_{i=1}^m \xi_i \Psi(\theta_i x) &= \sum_{i=1}^m \xi_i \sum_{n=0}^{\infty} a_n (\theta_i x)^n 
\\ &= \sum_{n=0}^{\infty} a_n \bigg( \sum_{i=1}^m \xi_i \theta_i^n \bigg) x^{n}, \quad x \in (-\eta,\eta)
\end{align*}
and we obtain
\begin{align*}
\sum_{n=0}^{\infty} a_n \bigg( \sum_{i=1}^m \xi_i \theta_i^n \bigg) x^{n} = 0, \quad x \in \Lambda(\mathbb{R}_+) \cap (-\eta,\eta).
\end{align*}
Since $\lambda \neq 0$ and $\Lambda(0) = 0$, there exists a sequence $(x_n)_{n \in \mathbb{N}} \subset \Lambda(\mathbb{R}_+) \cap (-\eta,\eta)$ with $x_n \neq 0$, $n \in \mathbb{N}$ and $x_n \rightarrow 0$. Therefore, the identity theorem for power series applies and yields
\begin{align*}
a_n \sum_{i=1}^m \xi_i \theta_i^n = 0, \quad n \in \mathbb{N}_0.
\end{align*}
Since $F(\mathbb{R}) \neq 0$ by assumption, relations (\ref{Taylor-series-Psi}), (\ref{psi-der-2}) show that $a_n > 0$ for every even $n \in \mathbb{N}$. It follows that
$B \xi = 0$, where $B \in \mathbb{R}^{m \times m}$ denotes the Vandermonde matrix $B_{ki} = \theta_i^{2k}$ for $k = 1,\ldots,m$ and $i = 1,\ldots,m$. Since $|\theta_i| \neq |\theta_j|$ for $i \neq j$, we deduce that $\xi_1 = \ldots = \xi_m = 0$, which proves
(\ref{dim-m}).
Since $m \in \mathbb{N}$ was arbitrary, we conclude (\ref{contra-2}), which contradicts Remark \ref{remark-U}. Consequently, $\sigma$ is constant on $h_0 + V$.

Now, let $h_0 \in H_{\beta}$ be arbitrary. Since $\mathcal{D}(d/dx)$ is dense in $H_{\beta}$, there exists a sequence $(h_n)_{n \in \mathbb{N}} \subset \mathcal{D}(d/dx)$ with $h_n \rightarrow h_0$. By the continuity of $\sigma$, for each $v \in V$ we obtain
\begin{align*}
\sigma(h_0) = \lim_{n \rightarrow \infty} \sigma(h_n) = \lim_{n \rightarrow \infty} \sigma(h_n + v) = \sigma(h_0 + v),
\end{align*}
showing that $\sigma$ is constant on $h_0 + V$. The second statement follows from relation (\ref{eqn-lin-ind}).
\end{proof}

\begin{remark}
The assumption $\dim W = 1$ implies that the volatility $\sigma$ is of the form $\sigma(h) = \Phi(h) \lambda$ with a real-valued mapping $\Phi : H_{\beta} \rightarrow \mathbb{R}$ and a function $\lambda \in H_{\beta'}^0$, whence we speak about constant direction volatility. Theorem \ref{thm-ddv} shows that in the presence of jumps we obtain restrictions on the mapping $\Phi$, which do occur in the Wiener driven case, see e.g. \cite{Bj_Sv, Bj_La, Tappe-Wiener}.
\end{remark}

Now, we assume that $\sigma = \phi \circ \ell$ with a continuous mapping $\phi : \mathbb{R} \rightarrow H_{\beta'}^0$ and a continuous linear functional $\ell : H_{\beta} \rightarrow \mathbb{R}$. We suppose that $\ell(W) = \mathbb{R}$.

\begin{corollary}\label{cor-ddv}
The following statements are true:
\begin{enumerate}
\item The volatility $\sigma$ is constant.

\item $\sigma$ is quasi-exponential, and we have $\langle (d/dx)^n \sigma : n \in \mathbb{N}_0 \rangle \subset V$.
\end{enumerate}
\end{corollary}

\begin{proof}
Since $W \subset V$ by Lemma \ref{lemma-ddv}, applying Theorem \ref{thm-ddv} with $h_0 = 0$ yields that the volatility $\sigma$ is constant on $W$. Note that $\ell|_W : W \rightarrow \mathbb{R}$ is an isomorphism. Therefore, for all $x,y \in \mathbb{R}$ we obtain
\begin{align*}
\phi(x) = \phi(\ell(\ell^{-1} x)) = \sigma(\ell^{-1} x) = \sigma(\ell^{-1} y) = \phi(\ell(\ell^{-1} y)) = \phi(y),
\end{align*}
showing that $\sigma$ is constant. The second statement follows from Theorem \ref{thm-ddv}.
\end{proof}

Now, we assume that in addition $\dim V = 1$. Then, according to Lemma \ref{lemma-ddv} we have $V = W$.

\begin{corollary}\label{cor-short-rate}
There are $\rho \in \mathbb{R}$, $\rho \neq 0$ and $\theta \in ( \beta' / 2,\infty)$ such that 
\begin{align}\label{vol-Vasicek}
\sigma \equiv \rho e^{-\theta \bullet}.
\end{align}
\end{corollary}

\begin{proof}
By Corollary \ref{cor-ddv}, the volatility $\sigma$ is constant, and we have $\langle (d/dx)^n \sigma : n \in \mathbb{N}_0 \rangle \subset V$. Since $\dim V = 1$, we obtain that (\ref{vol-Vasicek}) is satisfied for some $\rho \in \mathbb{R}$, $\rho \neq 0$ and $\theta \in \mathbb{R}$. 
By the Definition (\ref{def-norm}) of the norm $\| \cdot \|_{\beta'}$ we have
\begin{align*}
\int_{\mathbb{R}_+} |\lambda'(x)|^2 e^{\beta' x} dx < \infty,
\end{align*}
and, by the Definition (\ref{def-subspace}) of the subspace $H_{\beta'}^0$ we have
\begin{align*}
\lim_{x \rightarrow \infty} \lambda(x) = 0.
\end{align*}
We conclude that $\theta \in ( \beta' / 2,\infty)$, which finishes the proof.
\end{proof}

From the literature, see, e.g. \cite{Jeffrey, Bj_Sv,
Filipovic-Teichmann-royal}, it is well-known that for Wiener driven
interest rate models the following three types of affine short rate
realizations exist:
\begin{itemize}
\item The Ho-Lee model.

\item The Hull-White extension of the Vasi\u{c}ek model.

\item The Hull-White extension of the Cox-Ingersoll-Ross model.
\end{itemize}
The evaluation at the short end $\ell : H_{\beta}
\rightarrow \mathbb{R}$, $\ell(h) := h(0)$ is a continuous linear functional (see, e.g. \cite[Thm. 4.1]{Tappe-Wiener}). Thus, applying Corollary \ref{cor-short-rate} for the L\'{e}vy case with jumps, we recognize (\ref{vol-Vasicek}) as the Hull-White extension of the Vasi\u{c}ek model, whereas an analogon for the Hull-White extension of the Cox-Ingersoll-Ross model does not exist.

\begin{remark}
The Ho-Lee model would correspond to (\ref{vol-Vasicek}) with $\theta = 0$. Note that in our framework this volatility is even excluded in the Wiener case because of the technical reason that $\alpha_{\rm HJM} \notin H_{\beta}$. Indeed, for $\alpha_{\rm HJM} \in H_{\beta}$ one necessarily needs that $\lim_{x \rightarrow \infty} \sigma(x) = 0$, see relation (5.13) in \cite{fillnm}, which is not satisfied for $\sigma \equiv \rho$ with $\rho \neq 0$.
\end{remark}

If the volatility $\sigma$ is of the form (\ref{vol-Vasicek}), then the HJMM equation (\ref{HJMM}) has a one-dimensional realization, see Corollary \ref{cor-det}.
By a well-known technique (see, e.g. \cite[Prop. 2.8]{Tappe-Wiener}), we can choose the short rate $r_t(0)$ as state process, whence we speak about a short rate realization.

\begin{appendix}

\section{Results about power series with several variables}\label{app-series}

For the proof of Theorem \ref{thm-ddv-pre} we require some results about power series with several variables. Since these results were not immediately available in the literature, we provide self-contained proofs in this appendix.

\begin{lemma}\label{lemma-Cauchy}
Let $(a_k)_{k \in \mathbb{N}_0} \subset \mathbb{R}$ and $(b_l)_{l \in \mathbb{N}_0} \subset \mathbb{R}$ be sequences such that the series
$\sum_{k \in \mathbb{N}_0} a_k$ and $\sum_{l \in \mathbb{N}_0} b_l$ are absolutely convergent. Then, the series
$$\sum_{(k,l) \in \mathbb{N}_0^2} a_k b_l$$
is also absolutely convergent, and we have
\begin{align*}
\sum_{(k,l) \in \mathbb{N}_0^2} a_k b_l = \bigg( \sum_{k \in \mathbb{N}_0} a_k \bigg) \cdot \bigg( \sum_{l \in \mathbb{N}_0} b_l \bigg).
\end{align*}
\end{lemma}

\begin{proof}
This is a direct consequence of the Cauchy product formula for absolutely convergent series (see, e.g. \cite[Satz 8.3]{Forster}).
\end{proof}

In what follows, let $p \in \mathbb{N}$ be a positive integer. Let $K \subset \mathbb{R}^p$ be a compact subset. For a function $f : K \rightarrow  \mathbb{R}$ we define the supremum norm
\begin{align*}
\| f \|_K := \sup \{ |f(x)| : x \in K \}.
\end{align*}
We will need the following version of Weierstrass' criterion of uniform convergence.

\begin{lemma}\label{lemma-Weierstrass}
Let $f_n : K \rightarrow \mathbb{R}$, $n \in \mathbb{N}_0$ be functions such that
$\sum_{n=0}^{\infty} \| f_n \|_K < \infty$.
Then, the series $\sum_{n=0}^{\infty} f_n$ converges absolutely and uniformly on $K$ to a continuous function 
$$f : K \rightarrow \mathbb{R}, \quad f(z) = \sum_{n \in \mathbb{N}} f_n.$$
\end{lemma}

\begin{proof}
We can literally adapt the proof for functions with one variable, see e.g. \cite[Satz 21.2]{Forster}.
\end{proof}

For $x \in \mathbb{R}^p$ and $k \in \mathbb{N}_0^p$ we introduce the notation
\begin{align*}
x^k := x_1^{k_1} \cdot \ldots \cdot x_p^{k_p}.
\end{align*}
For $a \in \mathbb{R}^p$ and $r > 0$ let $K_r(a)$ be the compact ball
\begin{align}\label{comp-ball}
K_r(a) := \{ x \in \mathbb{R}^p : \| x - a \|_{\mathbb{R}^p} \leq r \}.
\end{align}

\begin{proposition}\label{prop-power-series}
Let $\pi : \mathbb{N}_0 \rightarrow \mathbb{N}_0^p$ be a bijective mapping, let $(c_n)_{n \in \mathbb{N}_0^p} \subset \mathbb{R}$ and $a \in \mathbb{R}^p$ be such that the power series
\begin{align}\label{power-series-single}
\sum_{\genfrac{}{}{0pt}{}{n = 0}{k = \pi(n)}}^{\infty} c_k (x-a)^k
\end{align}
converges for some $x \in \mathbb{R}^p$ with $x_i \neq a_i$ for all $i=1,\ldots,p$. Then, for all $0 < r < \min \{ |x_1-a_1|,\ldots,|x_p-a_p| \}$ the power series (\ref{power-series-single}) converges absolutely and uniformly on $K_r(a)$ to a continuous function 
\begin{align*}
f : K_r(a) \rightarrow \mathbb{R}, \quad f(z) = \sum_{k \in \mathbb{N}_0^p} c_k (z-a)^k.
\end{align*}
\end{proposition}

\begin{proof}
For each $k \in \mathbb{N}_0^p$ we define the function 
\begin{align*}
f_k : \mathbb{R}^p \rightarrow \mathbb{R}, \quad f_k(z) := c_k (z-a)^k. 
\end{align*}
Since the series (\ref{power-series-single}) converges, there exists a constant $M \geq 0$ such that 
\begin{align*}
|f_k(x)| \leq M \quad \text{for all $k \in \mathbb{N}_0^p$.} 
\end{align*}
Let $0 < r < \min \{ |x_1-a_1|,\ldots,|x_p-a_p| \}$ be arbitrary. We define the vector
\begin{align*}
\Theta := \bigg( \frac{r}{| x_1-a_1 |}, \ldots, \frac{r}{| x_p-a_p |} \bigg) \in (0,1)^p.
\end{align*}
For all $z \in K_r(a)$ and $k \in \mathbb{N}_0^p$ we obtain
\begin{align*}
|f_{k}(z)| &= | c_{k} (z-a)^{k} | = |c_{k} (x-a)^{k}| \frac{|(z-a)^{k}|}{|(x-a)^{k}|} 
\\ &= |f_k(x)| \frac{|z_1-a_1|^{k_1} \cdot \ldots \cdot |z_p-a_p|^{k_p}}{|x_1-a_1|^{k_1} \cdot \ldots \cdot |x_p - a_p|^{k_p}} 
\\ &\leq M \frac{r^{k_1} \cdot \ldots \cdot r^{k_p}}{|x_1-a_1|^{k_1} \cdot \ldots \cdot |x_p-a_p|^{k_p}} = M \Theta_1^{k_1} \cdot \ldots \cdot \Theta_p^{k_p} = M \Theta^{k}.
\end{align*}
By the geometric series and Lemma \ref{lemma-Cauchy}, the series
\begin{align*}
\sum_{k \in \mathbb{N}_0^p} \Theta^{k} = \prod_{i=1}^p \bigg( \sum_{k \in \mathbb{N}_0} \Theta_i^k \bigg)
\end{align*}
converges absolutely. Therefore, we obtain
\begin{align*}
\sum_{\genfrac{}{}{0pt}{}{n = 0}{k = \pi(n)}}^{\infty} \| f_k \|_{K_r(a)} < \infty,
\end{align*}
and hence, applying Lemma \ref{lemma-Weierstrass} concludes the proof.
\end{proof}

\end{appendix}

\end{document}